\documentclass[12 pt]{amsart}
\usepackage{amsmath,amssymb,amsthm}
\usepackage{longtable}
\usepackage{a4wide}
\newcommand{\GF}{\mathrm{GF}}

\newtheorem{conjecture}{Conjecture}[section]
\newtheorem{defn}{Definition}[section]
\newtheorem{theorem}[defn]{Theorem}

\newtheorem{lemma}[defn]{Lemma}

\def \d {\mathrm d}

\title{There is no upper bound for  the diameter of the commuting graph of a finite group}
\author{Michael Giudici }\author{Chris Parker}

\address{  School of Mathematics and Statistics\\
        The University of Western Australia\\
        35 Stirling Highway\\
		Crawley WA 6009, Australia}
	\email{michael.giudici@uwa.edu.au}
\address{
School of Mathematics\\
University of Birmingham\\
Edgbaston\\
Birmingham B15 2TT\\
United Kingdom} \email{c.w.parker@bham.ac.uk}

\begin{document}

\begin{abstract} We construct a family of finite special $2$-groups which have commuting graph of increasing diameter.
\end{abstract}
\maketitle

\section{Introduction}

For a group $G$, the \emph{commuting graph}  $\Gamma(G)$ of $G$ is the graph which has vertices the non-central elements of $G$ and two distinct vertices of $\Gamma(G)$ are adjacent if and only if they commute in $G$. In \cite{I},
 Iranmanesh and  Jafarzadeh conjecture that the commuting graph  of a finite group is either disconnected or has diameter bounded above by a constant independent of the group $G$.  They support this conjecture by proving that  the commuting graph of  $\mathrm{Sym}(n)$ and $\mathrm {Alt}(n)$ is either disconnected or has diameter  at most $5$. The conjecture is verified by the second author for the special case of soluble groups with trivial centre in \cite{P} where it is shown that the appropriate constant for such groups is $8$. This followed earlier work of Woodcock \cite{W} and Giudici  and Pope  \cite{G}. Further support for the conjecture is provided by the
work of Segev and Seitz  which  demonstrates that the commuting graph of a classical simple group defined over a field of order greater than $5$ is either disconnected or has  diameter at most $10$ and at least $4$ \cite[Corollary (pg. 127), Theorem 8]{SS}. In addition  they show  the commuting graph of the exceptional Lie type groups other than $\mathrm E_7(q)$ and the sporadic simple groups are disconnected \cite[Theorem 6]{SS}.
In  \cite{Heg} Hegarty and Zhelezov suggest a construction of a class of $2$-groups motivated by probabilistic methods aimed at providing a counter example to the    Iranmanesh and  Jafarzadeh conjecture. Though as yet unsuccessful, their putative examples motivated the examples presented in this article. Their supporting calculations yielded a group with commuting graph having diameter 10, the largest known diameter in the literature.  Our theorem is as follows.

\begin{theorem}\label{1} For all positive integers $d$, there exists a finite special $2$-group $G$ such that the commuting graph  of $G$ has diameter greater than $d$.
\end{theorem}

Theorem~\ref{1} proves that the  Iranmanesh and  Jafarzadeh conjecture is false. However, we believe that  it is most probably true that the commuting graph of a finite group with trivial centre is either disconnected or has diameter bound above by a constant.
 We are confident enough in this judgment to formulate it as a formal conjecture:
 \begin{conjecture}
There is an absolute constant $d$ such that if $G$ is a finite group with trivial centre, then the commuting graph of $G$ is either disconnected or has diameter at most $d$.
\end{conjecture}

We remark, that if the definition of the commuting graph of a group is revised so that the vertices of the graph are \emph{all} the non-trivial elements of $G$, then our conjecture is that the modified commuting graph of a finite group $G$ is either disconnected or has diameter bounded above by a constant independent of $G$.

\subsection*{Acknowledgement} The authors are grateful to the Banff International Research Station for supporting the Groups and Geometries conference in September 2012. The research reported in this article is a direct result of discussions  initiated at that research meeting. The first author's work on this research was supported by an ARC Discovery Project.

\section{The construction}

Suppose that $m$ is an integer with $m \ge 3$ and $V_m$  and $W_m$ are vector spaces defined over $\GF(2)$ which have dimension  $m$ and $m-2$ respectively.  Assume $x_1, \dots, x_m$
is an ordered basis for $V_m$ and $y_1, \dots, y_{m-2}$ is an ordered basis for $W_m$.
 Let $f_m : V_m \times V_m \rightarrow W_m$ be the bilinear map defined by the bilinear extension of the following map $$f_m(x_i,x_j) = \begin{cases} 0& j\in \{i, i+1\}\\  y_{j-i-1}&i+2 \le j \le m\\0&i > j \end{cases}.$$Because $f_m$ is  bilinear it is immediate that it is a $2$-cocycle.
  Therefore  we can define the group $H_m$ which has underlying set $V_m\times W_m$ and multiplication  defined as follows: for $(a,b), (c,d) \in H_m$,  $$(a,b)\cdot (c,d) = (a+c,f_m(a,c)+ b+d).$$ Then  $H_m$ is a central extension of $V_m$ by $W_m$. Furthermore, as $f_m$ is non-zero,  $H_m$ is a nilpotent group   of class $2$. Note that the identity of $H_m$ is $(0,0)$.

We  calculate $(x_i,0)(x_i,0) = (0,f_m(x_i,x_i))=(0,0)$ so that $(x_1,0), \dots, (x_m,0)$ are involutions and \begin{eqnarray*}[(x_i,0),(x_j,0)]&=& (x_i,0)(x_j,0)(x_i,0)(x_j,0)\\&=& \begin{cases} (0,f_m(x_i,x_j))= (0,y_{j-i-1})& i+1<j\\
(0,f_m(x_i,x_j))= (0,y_{i-j-1})&i>j+1\\(0,0)&\text{otherwise}\end{cases}.\end{eqnarray*}

The following lemma is elementary to prove.
\begin{lemma}\label{Hmfacts} Assume that $m \ge 4$. Then the following hold:
\begin{enumerate}
\item[(i)] We have $C_{H_m}((x_1,0) ) = \langle (x_1,0), (x_2,0), (0,w)\mid w \in W_m\rangle$ and $C_{H_m}((x_m,0) ) = \langle (x_m,0), (x_{m-1},0), (0,w)\mid w \in W_m\rangle$.
\item[(ii)]$[H_m,H_m] = Z(H_m) = \{(0,w)\in H_m\mid w \in W_m\}$ has order $2^{m-2}$.
\item [(iii)]$X = \{(v,0) \in H_m \mid v \in V_m\}$ is a transversal to $Z(H_m)$ in $H_m$.
\item [(iv)] $\langle (x_1,0), \dots, (x_{m-1},0)\rangle \cong  H_{m-1}$.
 \end{enumerate}
\end{lemma}

\begin{proof} To see (i), first note that $C_{H_m}((x_1,0) ) \ge \langle  (0,w)\mid w \in W_m\rangle$. Suppose that   $(v,0) \in C_{H_m}((x_1,0) ) $ with $v \in V_m \setminus \{0\}$.
Write $v= x_{i_1}+ \dots + x_{i_r}$ with $1 \le i_1\le \dots  \le i_r$. If $i_r\ge  3$, then $[(x_1,0),(v,0)]= (0,y_{i_1-2} + \dots + y_{i_r-2}) \ne (0,0)$. Thus   $i_r \le 2$ and this proves the first part of (i). The proof of the  second part is similar.

  Clearly $$[H_m, H_m] = \langle (0,y_{j-i-1}) \mid 1 \le i < j \le m \rangle = \{(0,w) \mid w \in W_m\}  \le Z(H_m).$$
On the other hand, by (i), as $m \ge 4$,  $$Z(H_m) \le C_{H_m}((x_1,0) ) \cap C_{H_m}((x_m,0) ) = \{(0,w) \mid w \in W_m\} .$$ So (ii) holds.

Parts (iii) and (iv) are obvious.
\end{proof}

We now commence with the investigation of the commuting graph $\Gamma_m = \Gamma(H_m)$ of $H_m$.   We define a subgraph $\Gamma_m^*$ of $\Gamma_m$. The vertices of $\Gamma_m^*$ are the non-trivial  elements of the transversal $X = \{(v,0) \in H_m \mid v \in V_m\}$ to $Z(H_m)$ and two elements of $X\setminus\{(0,0)\}$ are joined if and only if they commute.
Then $\Gamma_m$ is the lexicographic product of  ${\Gamma}_m^*$ and the complete graph on $|Z(H_m)|$ vertices  \cite{vahidi}. Thus the diameter of $\Gamma_m$ is equal to the diameter of ${\Gamma}^*_m$.  In particular,  to prove Theorem~\ref{1}, it suffices to prove that for every natural number $d$, there exists $m$ such that  $\Gamma_m^*$ has diameter greater than $d$. This is now our objective. To make the notation less unwieldy we abbreviate the elements $(x_i,0)$, $1 \le i \le m$,  by $x_i$ and $(0,y_i)$, $1 \le i \le m-2$, by $y_i$ expecting that no significant confusion will occur. We also set $Z_m = Z(H_m)$.

We know that $\Gamma_4^*$ has 15 vertices and  elementary calculations yield that it has a graphical representation  as follows:\begin{center}

\setlength{\unitlength}{.1in}%{.025in}
\begin{picture}(13, 15)

\put(2,3){\line(1,0){2}}
\put(4,6){\line(1,0){6}}
\put(4,3){\line(0,1){3}}
\put(2,3){\line(2,3){2}}
\put(4,6){\line(-3,1){3}}
\put(4,6){\line(-1,1){3}}
\put(1,7){\line(0,1){2}}
\put(4,6){\line(1,1){3}}
\put(10,6){\line(-1,1){3}}
\put(7,9){\line(-1,1){3}}
\put(7,9){\line(1,1){3}}
\put(7,9){\line(-3,1){3}}
\put(7,9){\line(3,1){3}}
\put(4,10){\line(0,1){2}}
\put(10,10){\line(0,1){2}}
\put(10,6){\line(1,1){3}}
\put(10,6){\line(3,1){3}}
\put(10,6){\line(2,-3){2}}
\put(10,6){\line(0,-1){3}}
\put(10,3){\line(1,0){2}}
\put(13,7){\line(0,1){2}}
\put(13,7){\circle*{.3}}
\put(13,9){\circle*{.3}}

\put(10,3){\circle*{.3}}
\put(10,6){\circle*{.3}}
\put(7,9){\circle*{.3}}
\put(4,10){\circle*{.3}}
\put(1,7){\circle*{.3}}
\put(4,6){\circle*{.3}}
\put(4,3){\circle*{.3}}
\put(2,3){\circle*{.3}}

\put(1,7){\circle*{.3}}
\put(1,9){\circle*{.3}}

\put(4,12){\circle*{.3}}
\put(12,3){\circle*{.3}}
\put(10,6){\circle*{.3}}
\put(10,12){\circle*{.3}}
\put(10,10){\circle*{.3}}

\end{picture}
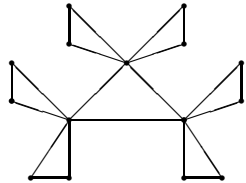
\end{center}\vspace{-5mm} \centerline{The graph $\Gamma_4^*$.}
Therefore $\Gamma_4^*$ is connected and has diameter $3$.

The proof that $\Gamma_m$ is connected only uses the fact that $\dim V_m - \dim W_m \ge 2$ and the connectivity of smaller graphs.
\begin{lemma}\label{connected} For all $m \ge 4$,
$\Gamma_m$ is connected.
\end{lemma}

\begin{proof}   We have already seen that $\Gamma_4^*$ is connected. Hence $\Gamma_4$ is connected. Assume $m > 4$  and  that $\Gamma_{m-1} $ is connected.

 Let $J = \langle x_1,\dots, x_{m-1}\rangle Z_m$. Then $J$ has index $2$ in $H_m$, $ \Gamma(J)$ is a subgraph of $\Gamma_m$ and $\Gamma(J) \cong \Gamma_{m-1} $ is connected.
Let $a \in H_m\setminus J$. It suffices to show  $C_{H_m}(a) \cap J \not \le Z_m$. This means we should show  $|C_{H_m}(a)/Z_m| \ge 4$.  The commutator map
$\phi:H_m/Z_m \rightarrow H_m'$  given by $bZ \mapsto [a,b]$ is a homomorphism
from  $H_m/Z_m \cong V_m$ of order $2^m$ to $H_m' \cong W_m $ which has order $2^{m-2}$
and so, indeed, $|C_{H_m}(a)/Z_m| \ge 4$ and therefore $\Gamma_m$  is connected.
\end{proof}

\begin{lemma} \label{unbound}Suppose that $d$ is an integer such that  $m>2^{d-1}$.  Assume that $w \in V(\Gamma_m^*)$ and $\mathrm d(x_1,w)= d$. If $x_n$ appears in the minimal expression for $w$, then $n \le 2^{d-1}+1$.
\end{lemma}

\begin{proof}  We have $C_{H_m}(x_1) = \langle x_1, x_2\rangle Z_m$ from Lemma~\ref{Hmfacts} (i). Hence $|C_{H_m}(x_1)/Z_m|=4$ and the vertices incident to $x_1$ in $\Gamma_m^*$   are $x_1+x_2$ and $x_2$. Thus  the highest subscript involved in vertices at distance 1 from $x_1$  is $2= 2^{1-1}+1$.  So the result is true for vertices at distance $1$ from $x_1$. Assume that the result is true for vertices at distance   $k$ from $x_1$.  Let $w \in V(\Gamma_m^*)$ be such that $d(x_1,w)=k+1$ and $u \in V(\Gamma_m^*)$  be incident to $w$ and have distance $k$ from $x_1$.  Write $w = x_{\beta_1} +\dots +x_{\beta_s}$ and $u= x_{\alpha_1} +\dots +x_{\alpha_r}$ where $\alpha_1 \le \dots \le \alpha_r$ and $\beta_1 \le \dots\le \beta_s$. Since $\d(x_1,u) = k$, $\alpha_r \le 2^{k-1}+1$. Because $[u,w]=(0,0)$, we have $$\sum_{{1\le i \le r}\atop {1 \le j \le s}} [x_{\alpha_i},x_{\beta_j}]= \sum_ {{1\le i \le r}\atop {1 \le j \le s}} y_{|\alpha_i-\beta_j| -1}=(0,0),$$ where we assume  $y_\ell$ with $\ell \le 0$ is  $(0,0)$.
If $\beta_s \le \alpha_r+1 \le 2^{k-1}+2$, then there is nothing to prove.  Hence we may assume that $\beta_s \ge  \alpha_r +2 \ge \alpha_1+2$. As $ \sum_ {{1\le i \le r}\atop {1 \le j \le s}} y_{|\alpha_i-\beta_j| -1} = (0,0)$, there exists $\alpha_1 \le \alpha_t \le \alpha_r$ and $\beta_1 \le \beta_u \le \beta_s$ such that $$y_{\beta_s-\alpha_1-1} = \begin{cases} y_{\alpha_t-\beta_u-1}& \alpha_t > \beta_u\\
y_{\beta_u-\alpha_t-1}& \alpha_t <\beta_u
\end{cases}.$$
Since $\beta_s \ge \beta _u$ and $\alpha_1 \le \alpha_t$, the latter possibility is impossible. Thus   $ {\beta_s-\alpha_1-1} = \alpha_t-\beta_u-1$ which means  $$\beta_s < \beta_s+ \beta_u = \alpha_1+\alpha_t \le 2 \alpha_r \le 2(2^{k-1}+1).$$
Therefore $\beta_s \le 2^k+1$ and the result follows by induction.
\end{proof}

\begin{proof}[Proof of Theorem~\ref{1}]   Lemmas~\ref{connected} and \ref{unbound} show that for any given integer $d$, there exists a positive integer $m$ such that $\Gamma_m$ is connected of diameter greater than $d$.

\end{proof}

One final remark: computations show that   for $4 \le m\le 16 $  the diameter of $\Gamma_m$ is  $m-1$.

\end{document}